\documentclass{amsart}

\usepackage{color}
\usepackage{amssymb,mathtools}
\usepackage{booktabs}
\usepackage{graphicx}
\usepackage[hidelinks]{hyperref}
\usepackage[capitalize]{cleveref}
\numberwithin{equation}{section}

\theoremstyle{plain}
\newtheorem{theorem}{Theorem}[section]
\newtheorem{proposition}[theorem]{Proposition}
\newtheorem{lemma}[theorem]{Lemma}

\theoremstyle{remark}
\newtheorem{remark}[theorem]{Remark}

\DeclareMathOperator{\dist}{dist}
\DeclareMathOperator{\dom}{dom}
\DeclareMathOperator{\rank}{rank}

\newcommand{\R}{\mathbb R}

\newcommand{\Sp}{\mathfrak S}

\newcommand{\Ha}{\mathcal H_a}

\begin{document}

\title[No Spectral Invisibility for Dissipative Barrier Truncations]{No Spectral Invisibility for Dissipative Barrier Truncations in Any Dimension}

\author{Matthew J. Colbrook}
\address{Department of Applied Mathematics and Theoretical Physics, University of Cambridge, Cambridge CB3 0WA, United Kingdom}
\email{mjc249@cam.ac.uk}

\author{Marco Marletta}
\address{School of Mathematics, Cardiff University, Senghennydd Road, Cathays, Cardiff CF24 4AG, United Kingdom}
\email{marlettam@cardiff.ac.uk}


\date{\today}


\begin{abstract}
The dissipative barrier method suppresses spectral pollution, but whether it can itself conceal genuine spectral points has remained open. Known as the graveyard problem in computational spectral theory, the higher-dimen\-sional case has remained unresolved for more than a decade. We resolve it for Schr\"{o}dinger operators in dimensions $d\geq2$; together with the known one-dimensional theorem, this settles the no-invisibility problem in all dimensions. Let $A=-\Delta+V$ be a Dirichlet Schr\"odinger operator on a connected open set $\Omega\subseteq\mathbb R^d$ with $V\in L^1_{\rm loc}(\Omega)$ bounded below, and set $H=A+iS$, where $S\geq0$ and $S\in L^p(\Omega)$, with $1<p<\infty$ for $d=2$ and $d/2\leq p<\infty$ for $d\geq3$. Let $(\Omega_R)_{R>0}$ be a nested family of nonempty connected bounded open sets such that $\Omega_R\nearrow\Omega$ as $R\nearrow+\infty$. Denote by $H_R$ the Dirichlet truncation of $H$ to $\Omega_R$. We prove that every spectral point of $H$ is detected by the truncations: for every $\lambda\in\sigma(H)$ and every neighborhood $U$ of $\lambda$, $\sigma(H_R)\cap U\neq\emptyset$ for sufficiently large $R$. Equivalently, $\sigma(H)\subseteq\liminf_{R\to\infty}\sigma(H_R)$. Thus the barrier method does not trade suppression of spectral pollution for spectral invisibility. No regularity of $\partial\Omega$ is required, and the assumptions reach the critical Sobolev scale. The proof combines compactness of the dissipative form perturbation, Cwikel-type Schatten estimates for Birman–Schwinger operators, generalized strong resolvent convergence, and a reverse Hansmann--Weyl spectral-variation inequality due to Gil’. A two-dimensional numerical example illustrates the absence of spectral invisibility for the truncated dissipative operators.
\end{abstract}

\maketitle

\section{Introduction}

There are two ways a spectral computation can lie. It can see spectral points that are not there, or fail to see genuine spectral points that are. The first phenomenon is spectral pollution: spectral points of finite-dimensional or truncated approximants accumulate outside the spectrum of the target operator. The second is spectral invisibility: points of the target spectrum are missed by these limiting approximations.\footnote{For a family of approximate operators $T_R$, absence of spectral invisibility means the lower spectral inclusion
$\sigma(T)\subseteq \liminf_{R\to\infty}\sigma(T_R)$, where $z\in\liminf_{R\to\infty}\sigma(T_R)$ if every neighborhood of $z$ meets $\sigma(T_R)$ for all sufficiently large $R$. Spectral pollution concerns the opposite inclusion, namely accumulation points of $\sigma(T_R)$ lying outside $\sigma(T)$.} These are not rounding errors, but failures of the approximation scheme itself  \cite{rappaz1977approximation,Levitin,davies2004spectral,colbrook2019compute,colbrook3}, \cite[Chapter 1]{colbrook2026infinite}.

For unbounded self-adjoint operators, especially in gaps of the essential spectrum, spectral pollution has been studied extensively \cite{rappaz1997spectral,MR2947291,lewin2009spectral,perschrodinger,MR1672271}. For example, it occurs in finite element approximations of Maxwell's eigenvalue problem but can be avoided using edge elements; see Boffi's review \cite[Section 20]{boffi2010finite}. This success motivated the development of discrete differential forms \cite{bossavit1988whitney,hiptmair2002finite} and the finite element exterior calculus \cite{arnold2006finite,arnold2010finite,arnold2018finite}. The non-self-adjoint problem is even more treacherous: there is no variational principle, and, moreover, one has to deal with pseudospectral effects \cite{trefethen2005spectra}.

The essential numerical range characterizes regions of possible spectral pollution, even in the non-self-adjoint setting. Its role in spectral approximation goes back to \cite{pokrzywa1981method}, and the corresponding framework for unbounded operators was developed in \cite{bogli2020essential}. Compression methods can pollute anywhere in the essential numerical range $W_{\mathrm{e}}(T)$. For domain truncation methods for PDEs on unbounded domains, pollution is contained in $W_{\mathrm{e}}(T)$ under mild hypotheses. Unfortunately, there is currently no analogous set that characterizes spectral invisibility.

A classical idea to improve spectral approximation for domain truncation is to avoid imposing a reflecting wall at the artificial boundary. The best-known realization is the perfectly matched layer (PML). In the frequency domain, PML can be viewed as a form of complex scaling, whose rigorous spectral theory goes back to Aguilar--Combes and Balslev--Combes \cite{aguilar1971class,balslev1971spectral}; the PML terminology and method originate with B\'erenger \cite{berenger1994perfectly}. PML is most often used for whole-space and exterior-domain problems, but it can be combined with half-space matching to treat more complicated geometries; see Bonnet-Ben Dhia et al. \cite{MR4366123}.

PML and complex scaling replace the original problem by a non-self-adjoint one in which the continuous spectrum is pulled away from the eigenvalues of interest; often the essential numerical range is displaced as well. These eigenvalues then become visible and can be computed without local pollution. The dissipative barrier method proposed in \cite{marletta2012eigenvalues} and developed abstractly in \cite{MR3255484} follows this philosophy: move the bad set rather than trust a reflecting truncation. These techniques are used in computational chemistry and have also been used to study quantum mechanical resonances \cite{nonnenmacher2009quantum,nonnenmacher2015decay} and Anderson-type localisation \cite{MR4550318}.

But this raises a sharp question: can a cure for pollution create invisibility? Could it hide, for example, part of the essential spectrum? In the self-adjoint case, under standard hypotheses, domain truncation does not exhibit this pathology, and the usual proofs rely on variational or Weyl sequence arguments. For non-normal operators that arise from these techniques, there is no obvious substitute.

\subsection{The graveyard problem}

In 2012, the second author and Naboko considered this question for dissipative Schr\"odinger operators
$$
H=-\Delta+V+iS\quad\text{in }L^2({\mathbb R}^d),
$$
where $V$ is real-valued, $-\Delta+V$ is self-adjoint, and $S\geq 0$ is sufficiently small at infinity to ensure
$
\sigma_{\rm ess}(H)=\sigma_{\rm ess}(-\Delta+V).
$
They proved in \cite{marletta2014finite} that, when $d=1$ and $S\in L^1\cap L^\infty$, domain truncation of $H$ cannot cause spectral invisibility. They also observed that the result should extend to higher dimensions and to $S\in L^p$, $p>1$, but that extension did not follow from their proof. The obstruction was concrete: part of the argument evaluated a resolvent kernel on the diagonal. This can be made to work in one dimension, but not higher.

A different door opened with Hansmann's 2013 theorem \cite{MR3054310}. Let $B$ be bounded and self-adjoint, let $C=B+K$, and suppose that $K$ lies in the Schatten class ${\mathfrak S}_r$ for some $r>1$. Then
\begin{equation}\label{hans_eq}
\sum_{\lambda\in\sigma_{\rm d}(C)} [\operatorname{dist}(\lambda,\sigma(B))]^r\leq C_r\|K\|_{{\mathfrak S}_r}^r.
\end{equation}
Here the eigenvalues in the discrete spectrum $\sigma_{\rm d}(C)$ are counted with algebraic multiplicity, and $C_r$ depends only on $r$, not on $B$, $C$, or the Hilbert space. The estimate in \cref{hans_eq} has become one of the basic tools behind non-self-adjoint 
Lieb-Thirring and related inequalities \cite{MR3016473,MR3521694,MR4525757,MR3177329}.

Marletta and Naboko recognized that Hansmann's theorem had the right scale. The intended route was to pass from $-\Delta+V$ and $H$ to bounded Birman--Schwinger operators $B$ and $C$. The barrier would satisfy the Sobolev-form condition $S\in L^p$, with $p\geq d/2$ for $d\geq3$ and $p>1$ for $d=2$. Truncating $-\Delta+V$ and $H$ on domains $\Omega_n\nearrow{\mathbb R}^d$ would then give pairs $B_n,C_n$, to which Hansmann's estimate could be applied with constants independent of $n$.

However, there is a catch. Hansmann's theorem measures how far eigenvalues of $C$ can move from $\sigma(B)$. The no-invisibility problem asks for the opposite protection: not that $C_n$ has too many eigenvalues, but that it cannot lose the right ones. If eigenvalues are missing, the corresponding terms disappear from the left-hand side of Hansmann's inequality; the estimate becomes easier, not harder, to satisfy. Thus what is needed is a Hansmann-type estimate with the roles of $B$ and $C$ reversed.

For several years, the suspected absence of spectral invisibility for dissipative Schr\"odinger operators has been discussed with experts on CLR bounds and the wider spectral community. The open question was whether the dissipative barrier truncations $H_R$ satisfy
$$
        \sigma(H)\subseteq\liminf_{R\to\infty}\sigma(H_R)
$$
in dimensions $d\geq2$ under the natural Sobolev-scale assumption on the absorbing potential $S$. The problem came to be known as the \textit{graveyard problem} in computational spectral theory: natural, important, and close enough to tempt experienced analysts, but with every serious approach dying at the same point. It was the sort of problem that could bury months of work and leave no theorem.

In this paper, we solve the graveyard problem and prove that dissipative barriers do not create spectral invisibility.

\subsection{Main Result}

Let $d\geq2$, and let $\Omega\subseteq\R^d$ be a nonempty connected open set. We are mainly interested in the cases in which $\Omega$ is unbounded, but our results cover bounded $\Omega$ too. We do not assume that $\partial\Omega$ has any regularity. Let $V\in L^1_{\rm loc}(\Omega)$ be real-valued and such that, for some $c_0\geq0$,
$
        V\geq -c_0
$
almost everywhere. By a shift of spectral parameter we may assume without loss of generality that $V \geq 1$.

Let $A=-\Delta+V$ be the self-adjoint Dirichlet Schr\"odinger operator\footnote{Writing $A=-\Delta+V$ is a common abuse of notation; we shall be clear when this becomes important. Here the hypothesis $V\in L^1_{\mathrm{loc}}$ is enough to ensure that $\dom a$ contains
$C_c^\infty$ functions. For the expression $-\Delta u + Vu$ to define an element of $L^2$ for $u\in C_c^\infty(\Omega)$
one needs $V\in L^2_{\mathrm{loc}}$.}
 defined by the closed, densely defined form
$$
 a[u,v]=\int_\Omega \nabla u\cdot\overline{\nabla v}
       +\int_\Omega Vu\overline v,
 \qquad
 \dom a=\{u\in H_0^1(\Omega):V^{1/2}u\in L^2(\Omega)\},
$$
and let $\| \cdot \|_a$ be the associated form norm
\begin{equation}\label{eq:shifted-form-norm}
 \|u\|_a^2=a[u,u] \geq \| u \|_{H^1_0(\Omega)}^2.
\end{equation}
In particular 
\begin{equation}\label{eq:domAhalf}
 \dom a = \dom(A^{\frac{1}{2}}), \;\;\; a[u,u] = \| A^{\frac{1}{2}}u \|_{L^2(\Omega)}^2, \;\;\; u \in\dom a. 
\end{equation}
Let $S\geq0$, $S\in L^p(\Omega)$, where
\begin{equation}\label{eq:exponent-range}
        1<p<\infty\quad(d=2),\qquad
        d/2\leq p<\infty\quad(d\geq3).
\end{equation}
Let
$$
        h[u,v]=a[u,v]+i\int_\Omega S u\overline v,
        \qquad \dom h=\dom a .
$$
The fact that $\sqrt{S}A^{-\frac{1}{2}}$ is compact, see \cref{lem:cwikel} below, implies that $h$ is closed and sectorial.
By Kato's first representation theorem for sectorial forms \cite[Chapter VI, Section 2]{kato2013perturbation}, $h$ defines an $m$-sectorial operator $H$; with the usual abuse of notation, $H=-\Delta + V + iS = A + iS$.

Let $(\Omega_R)_{R>0}$ be a nested family of nonempty connected bounded open sets such that $\Omega_R\nearrow\Omega$ as $R\nearrow+\infty$; for instance, we might choose a fixed point $y_0\in\Omega$ and let $\Omega_R=\Omega\cap \{y\in{\mathbb R}^d: |y-y_0|<R\}$, if these are connected.
Let
$$
 a_R[u,v]=\int_{\Omega_R}\nabla u\cdot\overline{\nabla v}
          +\int_{\Omega_R}Vu\overline v,\qquad \dom a_R=\{u\in H_0^1(\Omega_R):
        V^{1/2}u\in L^2(\Omega_R)\},
$$
and define
$$
 h_R[u,v]=a_R[u,v]+i\int_{\Omega_R}Su\overline v.
$$
Due to the Dirichlet boundary conditions on $\partial\Omega_R$, elements of the form domains $\dom a_R = \dom h_R$ 
extend by zero to elements of $\dom a = \dom h$. Hence $a_R$ and the associated positive self-adjoint $A_R$ satisfy
\begin{equation}\label{eq:shifted-form-norm-R}
a_R[u,u] = \| A_R^{1/2}u \|_{L^2(\Omega_R)}^2 \geq \| u \|_{H^1_0(\Omega_R)}^2, \;\;\;
 u \in \dom a_R = \dom(A_R^{1/2}).
\end{equation}
The $h_R$ are uniformly sectorial, as restrictions of $h$, and we denote by $H_R$ the associated $m$-sectorial operators
in $L^2(\Omega_R)$.

The following theorem solves the graveyard problem.

\begin{theorem}\label{thm:main}
For every $\lambda\in\sigma_{\mathrm{ess}}(H)$ and every open neighborhood $U$ of $\lambda$,
$$
        \sigma(H_R)\cap U\neq\emptyset
        \qquad\hbox{for all sufficiently large }R .
$$
The same is guaranteed for eigenvalues $\lambda\in\sigma_{\rm d}(H)$ if $H$ enjoys the unique continuation
principle, for instance if $V\in L^p_{\mathrm{loc}}(\Omega)$ for $p$ satisfying \eqref{eq:exponent-range}.
\end{theorem}

The theorem says that for every $\lambda\in\sigma(H)$ and every tolerance $\varepsilon>0$, all sufficiently large computational domains have at least one spectral point of $H_R$ within distance $\varepsilon$ of $\lambda$. In an actual computation one first replaces the unbounded problem by a problem on a finite computational domain, and then discretizes the finite-domain operator, for instance by finite elements, finite differences, or spectral methods; see \cref{sec:num_example} for an example. \cref{thm:main} concerns the first step.

\begin{remark}[Essential versus discrete spectra]
The proof of the statement in Theorem \ref{thm:main} concerning $\lambda$ in the discrete spectrum $\sigma_{\mathrm{d}}(H)$ 
will be obtained by using the abstract result \cite[Thm. 5.4 i)]{bogli2020essential}, following methods very similar to those used to prove \cite[Theorem 7.1 iii)]{bogli2020essential} for even-order elliptic operators with bounded coefficients. This strategy
works because the unique continuation principle 
\cite{JerisonKenig1985,SchechterSimon1980} ensures that eigenvalues of $H$ lie strictly in the upper half-plane, while the essential numerical range $W_{\mathrm{e}}(H)$ is equal to $W_{\mathrm{e}}(A)$, and hence real. The important contribution in Theorem \ref{thm:main} is therefore the case $\lambda\in \sigma_{\mathrm{ess}}(H)$.
\end{remark}

\begin{remark}[Easy extensions]
The condition that $V$ be bounded below can be removed as long as the negative part $V_{-}$ of $V$ lies in $L^p$, where $p$ satisfies (\ref{eq:exponent-range}), and $C_{p,d}\| V_{-} \|_{p} < 1$, where $C_{p,d}$ is the Cwikel constant appearing in \cref{lem:sobolev-form-bound} below. This ensures that (after a shift of spectral parameter, if necessary) $a[u,u] \geq c \| u \|_{H^1_0}^2$ for some $c>0$. The result also extends to higher-order operators; see \cref{rem:ho-basic}.
\end{remark}

\subsection{Route to the Result}\label{sec:route}

The proof can be viewed as an eigenvalue-counting stability argument for the computational domains $\Omega_R$. Near a point of the essential spectrum of the self-adjoint operator $A$, the operators $A_R$ have increasingly many eigenvalues. The dissipative barrier perturbs the corresponding resolvents by a uniformly Schatten-controlled non-self-adjoint term. A reversed Hansmann inequality then prevents all of these eigenvalues from being transported away from the neighborhood. This is the mechanism ruling out spectral invisibility.

The missing ingredient to prove \cref{thm:main} is supplied by two papers of Gil' from 2024 \cite{MR4764540,MR4722454}. Taken together, they give precisely the reversed form of Hansmann's estimate needed here. We found these papers by using MathSciNet to search for papers citing Hansmann's original article \cite{MR3054310}.

The reason for making that search is also worth recording. The first author (MJC) had built a system of ChatGPT-based agents to propose, test, and attack possible proofs of the no-invisibility theorem. After a number of false starts, including arguments with real gaps and arguments relying on results which do not exist, the agents converged on the strategy just described -- which Marletta and Naboko had already dismissed for lack of a reversed form of Hansmann's estimate. The final step consisted of two lemmas which, if true, would have implied a crude version of \cref{hans_eq} with $A$ and $B$ reversed. We did not trust those lemmas. But it seemed likely that the idea was a distorted echo of human-written mathematics in the training data. We therefore looked for the uncited human source. We did not find exactly that. Instead, we found the stronger results of Gil' \cite{MR4764540,MR4722454}, which give a cleaner proof.

\subsection{Notation and preliminaries}

For a compact operator $K$, let $\mu_1(K)\geq\mu_2(K)\geq\cdots$ denote its singular values, enumerated by multiplicity. For $1\leq s<\infty$, the Schatten class $\Sp_s$ is the set of compact $K$ whose singular values form a sequence in $\ell^s({\mathbb N})$, 
with norm $\| \cdot \|_{\Sp_s}$ defined by
$$
        \|K\|_{\Sp_s}^s=\sum_{j=1}^\infty \mu_j(K)^s 
$$
The weak Schatten class $\Sp_{s,\infty}$ consists of those compact $K$ such that
$$
        \|K\|_{s,\infty}^s=\sup_{\kappa>0}\kappa^s \operatorname{rank}{\bf 1}_{(\kappa,\infty)}(|K|)<\infty.
$$ 
We shall also use the Marcinkiewicz interpolation inequality \cite{MR80887} for
a bounded operator $K$ in a Schatten scale:
\begin{equation}\label{eq:weak-schatten-interpolation}
        \|K\|_{\Sp_t}^t\leq {\frac{t}{t-s}}\|K\|^{t-s}\|K\|_{s,\infty}^s,
        \qquad t>s .
\end{equation}

In \cref{sec:prelim} we recall the proof that the dissipative barrier $iS$ is form compact relative to $A$ so
that, in particular\footnote{Throughout, $\sigma_{\rm ess}$ denotes the index-zero Fredholm essential spectrum: $z\notin\sigma_{\rm ess}(T)$ means that $T-z$ is Fredholm of index zero. There are at least five sensible definitions of essential spectra \cite[Chapters I and IX]{edmunds1987spectral}. They all coincide for the operators considered in this paper.}
$
        \sigma_{\rm ess}(H)=\sigma_{\rm ess}(A).
$
We also prove Schatten class bounds for various Birman--Schwinger-type operators measuring the barrier, using a theorem of Cwikel. Although some of these proofs are standard, extra care is needed to show that the constants appearing in the inequalities can be chosen to be the same for the domains $\Omega_R$ as for $\Omega$. 

In \cref{sec:BS_trace} we establish uniform bounds on generalized resolvent differences:
$$
     \sup_{R>0} \|   H_R^{-1}-A_R^{-1} \|_{\Sp_r} < \infty, \qquad r>p.
$$
We also prove generalized strong resolvent convergence of $A_R$ to $A$. Quite general results of this type are known in various contexts \cite[Section 7]{bogli2020essential}. This implies that eigenvalues of $A_R$ must accumulate near $\sigma_{\mathrm{ess}}(A)$, which will be important.

In \cref{sec:spec_var}, we use Gil's theorem, which gives a reversed Hansmann-type estimate: a Schatten perturbation of a compact self-adjoint operator cannot make too many eigenvalues disappear. Finally, \cref{sec:final_proof} combines this with the standard self-adjoint spectral inclusion for the Dirichlet sections. Near a point of $\sigma_{\rm ess}(A)$, the self-adjoint sections have arbitrarily many eigenvalues. Gil's estimate implies that the dissipative sections cannot lose more than finitely many of these eigenvalues. That is the no-invisibility mechanism and allows us to prove \cref{thm:main}. \cref{sec:num_example} concludes with an illustrative example in two dimensions.

\section{Schatten class bounds}
\label{sec:prelim}

The estimates in this section are uniform in $R$, which allows the family $H_R$ to be interpreted as a stable sequence of computational problems. Let $p'=p/(p-1)$ ($\frac{1}{p}+\frac{1}{p'}=1$) and set $q=2p'$. If $d\geq3$ and $p\geq d/2$, then $q\leq 2d/(d-2)$. If $d=2$ and $p>1$, then $2\leq q<\infty$. The following lemma collects some standard facts.

\begin{lemma}[Sobolev form bound]\label{lem:sobolev-form-bound}
There is a constant $C_{d,p}$ such that for all $u,v\in\dom a$,
\begin{equation}\label{eq:2.1}
 \int_\Omega S |u|^2 \leq \|S\|_p\|u\|_{2p'}^2 \leq C_{d,p}\|S\|_p\|u\|_{H^1_0(\Omega)}^2  \leq C_{d,p}\|S\|_p\|u\|_a^2.
 \end{equation}
The same estimate holds on every $\Omega_R$ with the same constant $C_{d,p}$.
\end{lemma}

\begin{remark} We shall shortly show that 
$\sqrt{S}(-\Delta_{\Omega}+1)^{-1/2}$, $\sqrt{S}A^{-\frac{1}{2}}$, $\sqrt{S}(-\Delta_{\Omega_R}+1)^{-1/2}$ and $\sqrt{S}A_R^{-1/2}$ are $R$-uniformly bounded in $\Sp_r$, for $r>2p$, by a multiple of $\| S \|_p^{1/2}$. The only reason we need \cref{lem:sobolev-form-bound} is that we use the Marcinkiewicz inequality to bootstrap from a basic operator-norm bound and a bound in $\Sp_{2p,\infty}$.
\end{remark}

\begin{proof}
If $u\in H_0^1(\Omega)$, its zero extension to ${\mathbb R}^d$ belongs to $H^1({\mathbb R}^d)$ and preserves both $\|u\|_2$ and $\|\nabla u\|_2$. The Sobolev inequality on ${\mathbb R}^d$ and $H^1$-coercivity of $a$ give
$$
        \|u\|_q\leq C_{d,q}\| u \|_{H^1_0(\Omega)} \leq C_{d,q}\|u\|_a
$$
for $2\leq q\leq 2d/(d-2)$ when $d\geq3$, using interpolation between $L^2$ and the critical Sobolev exponent if necessary. When $d=2$, the same conclusion holds for $q\in[2,\infty)$ \cite[Section 5.6.1]{evans2010partial}. Taking $q=2p'\geq 2$ and applying H\"older gives \eqref{eq:2.1}. For $\Omega_R$ the identical argument is applied to the zero extension from $H_0^1(\Omega_R)$ to $H^1({\mathbb R}^d)$, so the constant is independent of $R$.
\end{proof}

\begin{lemma}\label{lem:cwikel-basic} The following weak Schatten class bounds hold:
\begin{equation}\label{sweak} 
\|\sqrt{S}(-\Delta_{\Omega_R}+1)^{-1/2}\|_{\Sp_{2p,\infty}}^2\leq \|\sqrt{S}(-\Delta_{\Omega}+1)^{-1/2}\|_{\Sp_{2p,\infty}}^{2}
 \leq C_{p,d} \|S\|_p. \end{equation}
Moreover for all $r>2p$, 
\begin{equation}\label{sstrong} \|\sqrt{S}(-\Delta_{\Omega}+1)^{-1/2}\|_{\Sp_{r}}^2,  \|\sqrt{S}(-\Delta_{\Omega_R}+1)^{-1/2}\|_{\Sp_{r}}^2
 \leq \tilde{C}_{r,d} \|S\|_p. \end{equation}
\end{lemma}
\begin{proof}
For $\Omega=\mathbb R^d$, this result is standard \cite[Chapter 4]{SimonTraceIdeals}, \cite{Cwikel}. We present details for the convenience of the reader and since our application requires that the constants can be chosen to be domain independent. All the operators under scrutiny are operator-norm bounded by a multiple of $\| S \|_p^{1/2}$, by \cref{lem:sobolev-form-bound}, so \eqref{sstrong} follows from \eqref{sweak} using \eqref{eq:weak-schatten-interpolation}. Thus it suffices to prove
\eqref{sweak}. 

If $\Omega\neq {\mathbb R}^d$ then extend $S$ isometrically by zero into $L^p({\mathbb R}^d)$. Let $\Delta_{{\mathbb R}^d}$ denote the Laplacian in 
$L^2({\mathbb R}^d)$. We start by proving 
\[ \|\sqrt{S}(-\Delta_{{\mathbb R}^d}+1)^{-1/2}\|_{\Sp_{2p,\infty}}^2 \leq C_{p,d} \|S\|_p. \]
For $g(\xi)=(|\xi|^2+1)^{-1/2}$, one has
\begin{equation}\label{eq:multiplier-distribution}
        |\{\xi:g(\xi)>t\}| =\omega_d(t^{-2}-1)_+^{d/2} \leq C_{d,p}t^{-2p},\;\;\; t>0 .
\end{equation}
Thus 
\[ \| g \|_{L^{2p,\infty}}^{2p} := \sup_{t>0} t^{2p} |\{\xi:g(\xi)>t\}| \leq C_{d,p}. \]
Since $\sqrt{S}\in L^{2p}({\mathbb R}^d)$, the Cwikel theorem (\cite[Theorem 4.2]{SimonTraceIdeals}; \cite{Cwikel}) gives
\begin{equation}\label{eq:free-cwikel-bound}
        \|\sqrt{S}(-\Delta_{{\mathbb R}^d}+1)^{-1/2}\|_{2p,\infty} \leq C_{d,p}\|\sqrt{S}\|_{2p} = C_{d,p}\|S\|_p^{1/2}.
\end{equation}
\cref{sweak} will now be proved if we can establish the domain monotonicity result
\[ \|\sqrt{S}(-\Delta_{\Omega_R}+1)^{-1/2}\|_{2p,\infty} \leq \|\sqrt{S}(-\Delta_{\Omega}+1)^{-1/2}\|_{2p,\infty} \leq  \|\sqrt{S}(-\Delta_{{\mathbb R}^d}+1)^{-1/2}\|_{2p,\infty}. \]
In fact, this is just domain monotonicity of the Dirichlet forms in disguise. 
To see this, define the Birman--Schwinger operator $B_{\Omega_R}$ by
\[ B_{\Omega_R} = (-\Delta_{\Omega_R}+1)^{-1/2}S(-\Delta_{\Omega_R}+1)^{-1/2}, \]
and likewise for $B_{\Omega}$ and $B_{{\mathbb R}^d}$. Note that
$$
\|\sqrt{S}(-\Delta_{\Omega_R}+1)^{-1/2}\|_{2p,\infty}^{2p}=\sup_{\rho>0}\rho^{2p} \operatorname{rank}{\bf 1}_{(\rho,\infty)}(B_{\Omega_R}^{1/2})=\sup_{\kappa>0}\kappa^{p} \operatorname{rank}{\bf 1}_{(\kappa,\infty)}(B_{\Omega_R}).
$$
By the Birman--Schwinger principle,
\[ \langle (B_{\Omega_R}-\kappa)u,u \rangle >0 \;\;\; \mbox{if and only if} \;\;\;
 \int_{\Omega_R}(S|v|^2 - \kappa(|\nabla v|^2 + |v|^2)) > 0, \]
where $v_R:=(-\Delta_{\Omega_R}+1)^{-1/2}u\in H^1_0(\Omega_R)$ extends by zero into a function $v$ in 
$H^1_0(\Omega)$. Taking $u_R = (-\Delta_{\Omega}+1)^{1/2}v$ gives   $\langle (B_{\Omega}-\kappa)u_R,u_R \rangle >0$. It follows that
$$
\operatorname{rank}{\bf 1}_{(\kappa,\infty)}(B_{\Omega_R})\leq \operatorname{rank}{\bf 1}_{(\kappa,\infty)}(B_{\Omega}).
$$
The remaining inequality follows similarly.
\end{proof}

Define $W_R^{\rm loc}$ and $W$ as the following operator closures:
\begin{equation}\label{eq:WRdef}
W_R^{\rm loc} = \overline{A_R^{-1/2}S_R A_R^{-1/2}},
\;\;\; W = \overline{A^{-\frac{1}{2}}SA^{-\frac{1}{2}}}.
\end{equation}
These closures are non-negative compact operators, e.g., from the formula
\[ W = \overline{A^{-\frac{1}{2}}SA^{-\frac{1}{2}}} = (\sqrt{S}A^{-\frac{1}{2}})^*(\sqrt{S}A^{-\frac{1}{2}}) \]
and compactness of $\sqrt{S}A^{-\frac{1}{2}}$ from \cref{lem:forms}.

\begin{lemma}\label{lem:cwikel} The following weak Schatten class bounds hold:
\begin{equation}\label{sweakw} 
\|W_R^{\rm loc}\|_{\Sp_{p,\infty}}  \leq \|W\|_{\Sp_{p,\infty}}
 \leq C_{p,d}^2 \|S\|_p. \end{equation}
Moreover for all $r>p$, 
\begin{equation}\label{sstrongw}
\|W_R^{\rm loc}\|_{\Sp_{r}} \leq \|W\|_{\Sp_{r}}
 \leq \tilde{C}_{r,d}^2 \|S\|_p. 
 \end{equation}
\end{lemma}
\begin{proof}

The inequalities 
$\|W_R^{\rm loc}\|_{\Sp_{p,\infty}} \leq \|W\|_{\Sp_{p,\infty}}$ and $\|W_R^{\rm loc}\|_{\Sp_{r}} \leq \|W\|_{\Sp_{r}}$
follow by a domain monotonicity argument similarly to \cref{lem:cwikel-basic}. Thus it suffices to prove
$\|W\|_{\Sp_{p,\infty}}\leq C_{p,d}^2 \|S\|_p$ and $\|W\|_{\Sp_{r}}\leq \tilde{C}_{r,d}^2 \|S\|_p$. Since $A^{-\frac{1}{2}}$ maps into
$H^1_0(\Omega)=\dom (-\Delta_{\Omega}+1)^{1/2}$,
\[ \sqrt{S}A^{-\frac{1}{2}}=[\sqrt{S}(-\Delta_{\Omega}+1)^{-1/2}][(-\Delta_{\Omega}+1)^{1/2}A^{-\frac{1}{2}}]. \]
The Schatten class operator $\sqrt{S}(-\Delta_{\Omega}+1)^{-1/2}$ enjoys the bounds in \cref{lem:cwikel-basic}. Since $\| A^{\frac{1}{2}}u \|^2 \geq \| u \|_{H^1_0(\Omega)}^2$ for
all $u\in\dom(A^{\frac{1}{2}})$, the bounded operator $(-\Delta_{\Omega}+1)^{1/2}A^{-\frac{1}{2}}$ has norm $\leq 1$. 
It follows that
$\|\sqrt{S}A^{-\frac{1}{2}}\|_{\Sp_{2p,\infty}}^{2}
 \leq C_{p,d} \|S\|_p$ where $C_{p,d}$ is the constant in \eqref{eq:free-cwikel-bound}.
From this we obtain $\| W \|_{\Sp_{p,\infty}}\leq C_{p,d}\| S \|_p$, and the remaining bound by Marcinkiewicz interpolation in \cref{eq:weak-schatten-interpolation}.
\end{proof}

\begin{lemma}\label{lem:forms}
The resolvent of $H$ satisfies
\begin{align}\label{eq:resdiff} H^{-1} &= A^{-\frac{1}{2}}(I+iW)^{-1}A^{-\frac{1}{2}},\\
 H^{-1}-A^{-1} &= A^{-\frac{1}{2}}(I+iW)^{-1}(-iW)A^{-\frac{1}{2}}.\label{eq:resdiff2}
\end{align}
Consequently, $H^{-1}-A^{-1}$ lies in $\Sp_{p,\infty}$ and in $\Sp_{r}$ for $r>p$; in particular $H^{-1}-A^{-1}$ is compact so $\sigma_{\rm ess}(H)=\sigma_{\rm ess}(A)$.
\end{lemma}
\begin{proof}
For all $u$ and $v$ in $\dom h = \dom a$ we have
\begin{equation}\label{eq:has} 
h[u,v] = a[u,v] + i\int_\Omega S u\overline v = \langle A^{\frac{1}{2}}u,A^{\frac{1}{2}}v\rangle + i\langle (\sqrt{S}A^{-\frac{1}{2}})A^{\frac{1}{2}}u,(\sqrt{S}A^{-\frac{1}{2}})A^{\frac{1}{2}}v\rangle. \end{equation}
Since $W=(\sqrt{S}A^{-\frac{1}{2}})^*(\sqrt{S}A^{-\frac{1}{2}})$,
\[ h[u,v] = \langle (I+iW)A^{\frac{1}{2}}u,A^{\frac{1}{2}}v\rangle. \]
Now take $u\in \dom H$ and $v=A^{-\frac{1}{2}}w\in \dom h = \dom a$ to obtain
\[ \langle Hu,A^{-\frac{1}{2}}w \rangle = \langle (I+iW)A^{\frac{1}{2}}u,w\rangle \]
whence for all $\phi,w\in L^2(\Omega)$, taking $u=H^{-1}\phi$,
\[  \langle A^{-\frac{1}{2}}\phi,w\rangle = \langle \phi,A^{-\frac{1}{2}}w \rangle = \langle (I+iW)A^{\frac{1}{2}}H^{-1}\phi,w\rangle. \]
This gives $A^{-\frac{1}{2}}= (I+iW)A^{\frac{1}{2}}H^{-1}$, from which (\ref{eq:resdiff}) follows; (\ref{eq:resdiff2}) is then
immediate by the resolvent identity.
\end{proof}

\begin{remark}[Higher-order operators]\label{rem:ho-basic}
The above Schatten estimates extend to positive elliptic realizations
\(A_D\) of order \(2m\), \(D\in\{\Omega,\Omega_R\}\), whose form domains are
\(H_0^m(D)\) and which are uniformly \(H_0^m\)-coercive. If
\(S\geq0\), \(S\in L^p(\Omega)\), \(p>1\), and \(2mp\geq d\), then, with
$
 T_D=\sqrt{S|_D}\,A_D^{-1/2}$ and $W_D=T_D^*T_D$,
one has
\[
 \|T_D\|_{\Sp_{2p,\infty}}
 \leq C\|S\|_p^{1/2},
 \qquad
 \|W_D\|_{\Sp_{p,\infty}}
 \leq C^2\|S\|_p,
\]
uniformly in $D$.
Indeed, the multiplier \((1+|\xi|^{2m})^{-1/2}\) belongs to
\(L^{2p,\infty}(\mathbb R^d)\), and the result follows from Cwikel's
theorem and uniform coercivity. Interpolation also gives
\(W_D\in\Sp_r\) uniformly for every \(r>p\).
\end{remark}

\section{Resolvent differences and strong resolvent convergence}
\label{sec:BS_trace}

\begin{lemma}\label{lem:kr}
Define $K_R= H_R^{-1}- A_R^{-1}$. For every $r>p$, one has
\begin{equation}\label{eq:k-schatten-convergence}
        \sup_R\|K_R\|_{\Sp_r}<\infty .
\end{equation}
For every finite $R>0$, $A_R^{-1},H_R^{-1}\in\Sp_r$.
\end{lemma}

\begin{proof}
As for \cref{eq:resdiff2}, one has
\[ K_R = H_R^{-1}-A_R^{-1} = A_R^{-\frac{1}{2}}(I+iW_R^{\rm loc})^{-1}(-iW_R^{\rm loc})A_R^{-\frac{1}{2}}. \]
Since the operator norms $\| A_R^{-\frac{1}{2}}\|$ and $\|(I+iW_R^{\rm loc})^{-1}\|$ are $\leq 1$,
\cref{lem:cwikel} yields
$$
        \|K_R\|_{\Sp_r}
        \leq \|W_R^{\mathrm{loc}}\|_{\Sp_r}
        \leq \tilde{C}_{r,d}^2 \| S \|_p.
$$
Finally we show that for finite $R>0$, $A_R^{-1}$ and $H_R^{-1}$ lie in $\Sp_r$ (note: their individual $\Sp_r$ norms
are not uniformly bounded in $R$). Obviously it suffices to show that $A_R^{-1}\in \Sp_r$. This is rather standard. We have assumed that, 
in the form sense, $A_R \geq -\Delta_{\Omega_R}+1$; thus it suffices to show that $(-\Delta_{\Omega_R})^{-1} \in \Sp_r$. This is well known: 
Weyl's law bounds the number of eigenvalues of $-\Delta_{\Omega_R}$ below $\lambda$ by
\[ N(\lambda,-\Delta_{\Omega_R}) \sim \frac{\omega_d}{(2\pi)^d} |\Omega_R| \lambda^{d/2}; \]
this implies that  $(-\Delta_{\Omega_R})^{-1}$ lies in $\Sp_{d/2,\infty}$ and hence in $\Sp_r$ for $r>p\geq d/2$.
\end{proof}

\begin{lemma} \label{lem:sconverge} Let $J_R$ denote extension by zero from $L^2(\Omega_R)$ to $L^2(\Omega)$
and let $Q_R$ denote restriction from $L^2(\Omega)$ to $L^2(\Omega_R)$. Then $A_R $ converges to $A$ and
$H_R$ converges to $H$ in generalized strong resolvent sense: for each $f\in L^2(\Omega)$, 
\[ J_RA_R^{-1}Q_Rf \to A^{-1}f,  \;\;\; J_RH_R^{-1}Q_Rf \to H^{-1}f, \qquad R \to +\infty. \]
In particular for each $\lambda\in \sigma_{\mathrm{ess}}(A)$ and each open interval $I$ containing $\lambda$,
\begin{equation}\label{eq:count} \operatorname{rank}{\bf 1}_I(A_R) \to + \infty, \;\;\; R \to +\infty. 
\end{equation}
\end{lemma}

\begin{proof}
Theorem 7.1 i) of \cite{bogli2020essential} proves generalized strong resolvent convergence
for even-order operators, though for ease of presentation the authors assumed bounded coefficients.
We outline here a different argument for the convenience of the reader. 

Let $\widetilde{\mathcal V}_R=J_R\,\,\mathcal \dom a_R \subset\dom a$. The union of these spaces is dense in $\dom a$ 
because $C_c^\infty(\Omega)$ is dense in $\dom a$.
For $f\in L^2(\Omega)$, let $u=A^{-1} f$ and let $u_R=J_R A_R^{-1}Q_Rf$. Then
$$
        a[u,v]=\langle f,v\rangle\quad(v\in\Ha),\qquad a[u_R,v_R]=\langle f,v_R\rangle\quad(v_R\in\widetilde{\mathcal V}_R),
$$
whence $a[u-u_R,v_R]=0$ for $v_R\in\widetilde{\mathcal V}_R$ giving
$\| u - u_R \|_a^2 = a[u-u_R,u-u_R] = a[u-u_R,u-w_R]\leq \| u - u_R\|_a \|u-w_R\|_a$ for all $w_R\in\widetilde{\mathcal V}_R$. 
This implies
\[  \|u-u_R\|_a \leq \inf_{w_R\in\widetilde{\mathcal V}_R} \| u - w_R\|_a. \]
By density of $C_c^\infty$ in $\dom a$ this gives $u_R\to u$ in $\|\cdot\|_a$, and in $L^2(\Omega)$ a fortiori. 
This proves $A_R\to A$ in generalized strong resolvent sense.

The proof that $H_R\to H$ (gsr) is similar. With $\tilde{u}=H^{-1}f$ and $\tilde{u_R} = J_RH_R^{-1}Q_Rf$ one obtains
$h[u-u_R,u-u_R] = h[u-u_R,u-w_R]$ for all $w_R \in \widetilde{\mathcal V}_R$. Thus by \eqref{eq:has}
\[ \| u - u_R\|_a^2 = \Re h[u-u_R,u-u_R] \leq \| u - u_R \|_a(1+\|\sqrt{S}A^{-\frac{1}{2}}\|^2)\| u - w_R \|_a, \]
whence $\| u - u_R\|_a\leq(1+\|\sqrt{S}A^{-\frac{1}{2}}\|^2) \displaystyle{\inf_{w_R\in\widetilde{\mathcal V}_R}} \| u - w_R\|_a$, 
differing from the self-adjoint case only by a constant.

Finally, to prove \eqref{eq:count}, we may assume without loss of generality that $0\not\in I$ and apply 
Theorem VIII.24 of \cite{reedsimon1}: since $A_R\to A$ in generalized strong resolvent sense, the spectral 
projections ${\bf 1}_{I}(A_R)$ converge to ${\bf 1}_{I}(A)$ in generalized sense. In particular since ${\bf 1}_{I}(A)$ is of 
infinite rank, the number of eigenvalues of $A_R$ in $I$ cannot remain bounded.
\end{proof}

\section{Spectral Variation}
\label{sec:spec_var}
The following result converts Schatten control of the non-self-adjoint perturbation into an eigenvalue-counting statement. In the proof of the main theorem, it says that a uniformly controlled dissipative perturbation cannot remove all bounded-domain eigenvalues from a spectral window. It is where Gil's work enters in a compact-operator form.

\begin{lemma}\label{gil_consequence}
Let $r>1$. Let $B$ and $C$ be compact operators on a separable Hilbert space, with $B=B^*$ and $C-B\in\Sp_r$. Let $(b_j)_{j=1}^\infty$ and $(c_j)_{j=1}^\infty$ be extended enumerations\footnote{Meaning that every nonzero eigenvalue is included with algebraic multiplicity and that the remaining entries are zero.} of the eigenvalues of $B$ and $C$. Then there is a permutation $\pi$ of $\mathbb N$ and a constant $b_r$, depending only on $r$, such that
\begin{equation}\label{eq:full-variation}
        \left(\sum_{k=1}^\infty |c_k-b_{\pi(k)}|^r\right)^{1/r}
        \leq (2+2b_r)\|C-B\|_{\Sp_r}.
\end{equation}
Consequently, if $J\subset{\mathbb R}$ is a Borel set and
$$
        \eta=\operatorname{dist}(J,\sigma(C)\cup\{0\})>0,
$$
then, with $N=\operatorname{rank}{\bf 1}_J(B)$,
\begin{equation}\label{eq:counting-gil}
        N\eta^r\leq (2+2b_r)^r\|C-B\|_{\Sp_r}^r .
\end{equation}
\end{lemma}

\begin{proof}
Gil's theorem \cite[Theorem 1.1]{MR4764540}, applied to the bounded operator $C$ and the self-adjoint operator $B$, gives
\begin{equation}\label{eq:real-variation}
        \left(\sum_{k=1}^\infty |\operatorname{Re}c_k-b_{\pi(k)}|^r\right)^{1/r}
        \leq (1+2b_r)\|C-B\|_{\Sp_r}.
\end{equation}
For the imaginary parts, we use Weyl's classical inequality,
\begin{equation}\label{eq:imag-variation}
        \left(\sum_{k=1}^\infty |\operatorname{Im}c_k|^r\right)^{1/r}
        \leq \|\operatorname{Im}C\|_{\Sp_r},
\end{equation}
for example, \cite[Chapter II, Section 6]{GohbergKreinIntro}. (Gil's paper \cite{MR4722454} gives recent refinements of this estimate in Schatten classes.) Since $B=B^*$, $\operatorname{Im}C=\operatorname{Im}(C-B)$. For any $r>1$ and 
operator $T\in \Sp_r$,
$$
        \|\operatorname{Im}T\|_{\Sp_r}
        =\left\|{\frac{T-T^*}{2i}}\right\|_{\Sp_r}
        \leq {\frac{1}{2}}\bigl(\|T\|_{\Sp_r}+\|T^*\|_{\Sp_r}\bigr)
        =\|T\|_{\Sp_r}.
$$
Thus $\|\operatorname{Im}C\|_{\Sp_r}\leq\|C-B\|_{\Sp_r}$. Combining this with \cref{eq:real-variation,eq:imag-variation}, and then using the triangle inequality in $\ell^r$ gives \cref{eq:full-variation}.

It remains only to prove the counting consequence. Since $\eta>0$, the set $J$ is separated from $0$, and ${\bf 1}_J(B)$ has finite rank. Each eigenvalue of $B$ in $J$, counted with multiplicity, occurs as $b_{\pi(k)}$ for exactly one $k$. For those $k$, the corresponding extended eigenvalue $c_k$ lies in $\sigma(C)\cup\{0\}$, so $|c_k-b_{\pi(k)}|\geq\eta$. Thus
$$
        N\eta^r\leq \sum_{k=1}^\infty |c_k-b_{\pi(k)}|^r,
$$
and \cref{eq:counting-gil} follows from \cref{eq:full-variation}.
\end{proof}

\section{Spectral Inclusion for the Barrier Truncations}
\label{sec:final_proof}

We now prove the lower spectral inclusion for the operators $H_R$.
For eigenvalues, this follows quickly using \cite[Thm. 5.4]{bogli2020essential}, as explained in Proposition
\ref{prop:discrete}. Points of essential spectrum are the real novelty, and are covered in Proposition \ref{prop:essential}.

\begin{proposition}[Discrete spectrum]\label{prop:discrete} Let $\lambda$ be an eigenvalue of $H$. If $H$ enjoys the
unique continuation principle -- for instance, if $V\in L^p_{\mathrm{loc}}(\Omega)$ -- then for each sequence $R_n\nearrow +\infty$ 
there exists, for each $n\in\mathbb N$, an eigenvalue $\lambda_n$ of $H_{R_n}$ such that $\lambda_n\to \lambda$ as 
$n\to\infty$.
\end{proposition}

\begin{proof} 
This proposition is closely related to \cite[Theorem 7.1 iii)]{bogli2020essential} which covers operators of any even order; 
however for simplicity of exposition that result was written in a form that would require our potential $V$ and barrier $S$ 
to be bounded. We therefore explain why these hypotheses can be discarded in the current setting.

The underlying abstract result guaranteeing spectral inclusion is  \cite[Thm. 5.4]{bogli2020essential}, which establishes
visibility of all spectral points lying outside the union of the limiting essential spectra 
\[ \sigma_{\rm ess}((H_{R_n})_{n\in\mathbb N})\cup \sigma_{\rm ess}((H_{R_n}^*)_{n\in\mathbb N})^* \]
provided that $H_{R_n}\to H$ and  $H_{R_n}^*\to H^*$ in generalized strong resolvent sense. Lemma \ref{lem:sconverge} 
gives this convergence, so our result follows from the following:

{\bf (1)} Eigenvalues are not real: if $Hu=\lambda u$ then
\[ \Im(\lambda)\|u \|^2 = \int_\Omega S|u|^2 > 0, \]
with strict inequality due to non-triviality of $S\geq 0$ and the unique continuation principle \cite{JerisonKenig1985,SchechterSimon1980}. 

{\bf (2)} The set that we do not want $\lambda$ to lie in, $\sigma_{\rm ess}((H_{R_n})_{n\in\mathbb N})\cup \sigma_{\rm ess}((H_{R_n}^*)_{n\in\mathbb N})^*$, is real. To see this observe first that there is a mild simplification $\sigma_{\rm ess}((H_{R_n})_{n\in\mathbb N}) = \sigma_{\rm ess}((H_{R_n}^*)_{n\in\mathbb N})^*$ due to the $J$-selfadjointness of the $H_{R_n}$. Next, $\sigma_{\rm ess}((H_{R_n})_{n\in\mathbb N})$ is contained in the limiting essential numerical range $W_{\rm e}((H_{R_n})_{n\in\mathbb N})$. Because $\langle H_{R_n}u,u\rangle = h[u,u]$ for all $u\in \dom\, H_{R_n}$, the hypotheses of \cite[Prop. 5.7]{bogli2020essential} are satisfied so
\[ W_{\rm e}((H_{R_n})_{n\in\mathbb N})\subseteq W_{\rm e}(H); \]
but the essential numerical range $W_{\rm e}(H)$ coincides with $W_{\rm e}(A)$ by \cite[Theorem 4.7]{bogli2020essential} because $A^{-\frac{1}{2}}SA^{-\frac{1}{2}}$ is compact. Since $A=A^*$ we have, in summary,
\[ \sigma_{\rm ess}((H_{R_n})_{n\in\mathbb N})\cup \sigma_{\rm ess}((H_{R_n}^*)_{n\in\mathbb N})^* \subset W_{\rm e}((H_{R_n})_{n\in\mathbb N})\subseteq W_{\rm e}(H) = W_{\rm e}(A) \subseteq \mathbb R \not\ni\lambda. \]
The proof is complete.
\end{proof}

\begin{proposition}[Essential spectrum]\label{prop:essential}
If $\lambda\in\sigma_{\rm ess}(H)$, then every open neighborhood $U$ of $\lambda$ meets $\sigma(H_R)$ for all sufficiently large 
$R>0$.
\end{proposition}

\begin{proof}
By \cref{lem:forms}, $ \sigma_{\rm ess}(H)=\sigma_{\rm ess}(A)$. Thus $\lambda\in\sigma_{\rm ess}(A)\subset[1,\infty)$, 
in particular, $\lambda\geq 1$. Suppose, for a contradiction, the assertion is false. Then there exist $R_n\to\infty$ such that $\sigma(H_{R_n})\cap U=\emptyset$. Choose a bounded open interval $I$ with
$$
        \lambda\in I,\qquad \overline I\subset U\cap(0,\infty).
$$
Let $\Phi(z)=z^{-1}$ with $\Phi(\infty)=0$, and set
$$
        \mathcal J=\Phi(\overline I),\qquad
				\Sigma=\{z\in{\mathbb C}:\operatorname{Re}z\geq 1,\ \operatorname{Im}z\geq0\}.
$$
Since the numerical range of $H_{R_n}$ is contained in $\Sigma$, so is $\sigma(H_{R_n})$. The pole of $\Phi$ at $z=0$ lies outside both $\Sigma$ and $\overline I$. Hence, $\overline I$ and $(\Sigma\setminus U)\cup\{\infty\}$ are disjoint compact sets of $\mathbb{C}\cup\{\infty\}$, and their images under $\Phi$ are separated:
\begin{equation}\label{eq:ess-separation}
        \eta:=\operatorname{dist}\bigl(\mathcal J,\Phi(\Sigma\setminus U \cup\{\infty\})\bigr)>0 .
\end{equation}
Let $B_n=A_{R_n}^{-1}$ and $C_n=H_{R_n}^{-1}$. Since $\sigma(H_{R_n})\cap U=\emptyset$, spectral mapping gives
$$
        \sigma(C_n)\cup\{0\}\subset \Phi((\Sigma\setminus U)\cup\{\infty\}),
$$
and therefore $\operatorname{dist}(\mathcal J,\sigma(C_n)\cup\{0\})\geq\eta$.

Choose $r>p$. By \cref{lem:kr}, $\sup_n\|B_n-C_n\|_{\Sp_r}<\infty$. Let
$$
        N_n=\rank{\bf 1}_{\mathcal J}(B_n).
$$
Since $\mathcal J=\Phi(\overline I)$ and $0\notin\mathcal J$, spectral mapping for the self-adjoint resolvent gives
\begin{equation}\label{eq:ess-rank-lower}
        N_n\geq \rank{\bf 1}_I(A_{R_n}).
\end{equation}
By \cref{lem:sconverge}, the right-hand side of \cref{eq:ess-rank-lower} tends to infinity, since $\lambda\in\sigma_{\rm ess}(A)$. 
On the other hand, \cref{gil_consequence} gives
$$
        N_n\eta^r\leq (2+2b_r)^r\sup_{n}\|B_n-C_n\|_{\Sp_r}^r<\infty .
$$
This is a contradiction. Hence no bad sequence of radii exists.
\end{proof}

\begin{proof}[Proof of \cref{thm:main}]
The isolated spectral points are covered by \cref{prop:discrete}, while the essential spectrum is covered by \cref{prop:essential}.
\end{proof}

The theorem is a guarantee for the domain-truncation stage of the dissipative barrier method. The next example illustrates this guarantee after discretizing the finite-domain problems: as the computational radius grows, spectral points of the truncated dissipative operators appear near the target spectrum rather than disappearing into the barrier.

\section{Numerical Example in Two Dimensions}
\label{sec:num_example}

The following computation illustrates \cref{thm:main}. We use a two-dimensional periodic Schr\"odinger operator with a localized defect: the self-adjoint part has band spectrum and a defect state in the first gap, while the dissipative finite sections recover the real bands and lift the defect state into the upper half-plane.

\subsection{The operator}
Let
$$
        V_{\rm per}(x,y)=\frac{1}{2}(2-\cos x-\cos y),\qquad
        \Upsilon(x,y)=\frac{1}{2}\exp(-(x^2+y^2)/1.8^2),
$$
and choose the dissipative barrier as follows. With $\rho=(x^2+y^2)^{1/2}$, set
$$
\chi(\rho)=\xi((\rho-8)/2),\qquad
\xi(t)=
\begin{cases}
0, & t\leq 0,\\[2mm]
\dfrac{\exp(-1/t)}
{\exp(-1/t)+\exp(-1/(1-t))}, & 0<t<1,\\[3mm]
1, & t\geq 1 .
\end{cases}
$$
We define the barrier
$$
        S(x,y)=\frac{1}{5}[1-\chi(\rho)+\chi(\rho)(1+\rho^2)^{-3/4}].
$$
Thus $S=0.2$ on $\rho\leq8$, while $S=0.2(1+\rho^2)^{-3/4}$ on $\rho\geq10$. Note that $S\geq0$, $S\in L^2(\mathbb R^2)$, and $S\notin L^1(\mathbb R^2)$. This choice deliberately illustrates the
$L^2(\mathbb R^2)\setminus L^1(\mathbb R^2)$ general assumptions covered by
\cref{thm:main}.

We consider
$$
        H=-\Delta+V_{\rm per}+\Upsilon+iS
        \qquad\hbox{on }L^2(\mathbb R^2).
$$
Set
\[
 L=-\frac{\mathrm d^2}{\mathrm dt^2}
      +\frac12(1-\cos t),
 \qquad
 A_{\rm per}=-\Delta+V_{\rm per}.
\]
Then
\[
 A_{\rm per}=L\otimes I+I\otimes L,
 \qquad
 \sigma(A_{\rm per})=\sigma(L)+\sigma(L).
\]
A calculation of the periodic and antiperiodic Mathieu eigenvalues
gives the first two one-dimensional bands
\[
 J_1=[0.386215349\ldots,0.472437796\ldots],
 \qquad
 J_2=[0.964777018\ldots,1.479256193\ldots].
\]
The pairwise band sums therefore give $\sigma(A_{\rm per})
   =\mathcal I_1\cup\mathcal I_2$ with
\[\mathcal I_1=[0.772430698\ldots,0.944875592\ldots],
 \qquad
 \mathcal I_2=[1.350992367\ldots,\infty).
\]
Since $\Upsilon$ and $S$ are bounded and tend to zero at infinity,
they are relatively compact perturbations of $A_{\rm per}$, and hence
\[
 \sigma_{\rm ess}(H)=\sigma(A_{\rm per})
                    =\mathcal I_1\cup\mathcal I_2.
\]
The parameters in $\Upsilon$ are chosen so that the self-adjoint
operator
\[
 A=A_{\rm per}+\Upsilon
\]
has a localized defect eigenstate in the first gap
$(0.944875592\ldots,1.350992367\ldots)$, as confirmed numerically below.

\subsection{Domain truncation and further discretisation}

For the numerical computations we specialise the domain truncations to the
squares
\[
        \Omega_R=(-R,R)^2,
        \qquad R\in\{10,12,14,16,18,20\},
\]
with homogeneous Dirichlet conditions on $\partial\Omega_R$.  For
$
        \delta\in\{\frac12,\frac14,\frac18,\frac1{16}\},
$
we partition $\Omega_R$ into axis-parallel squares of side length $\delta$,
insert the centre of every square, and join it to the four vertices.  This
gives a nested, shape-regular criss-cross triangulation
$\mathcal T_{R,\delta}$ satisfying
$
        \max_{T\in\mathcal T_{R,\delta}}\operatorname{diam}(T)=\delta .
$
Let $X_{R,\delta}\subset H^1_0(\Omega_R)$ be the corresponding conforming
space of continuous piecewise affine functions.  The Galerkin eigenproblem
is to find $z\in\mathbb C$ and $0\ne u_\delta\in X_{R,\delta}$ such that
\[
 \int_{\Omega_R}
 \left(
   \nabla u_\delta\cdot\overline{\nabla v_\delta}
   +(V_{\rm per}+\Upsilon+iS)u_\delta\overline{v_\delta}
 \right)
 =z\int_{\Omega_R}u_\delta\overline{v_\delta},
 \qquad v_\delta\in X_{R,\delta}.
\]
We denote the resulting finite-dimensional Galerkin operator by
$H_{R,\delta}$. To indicate the scale of the computation, for $R=20$ and
$\delta=1/16$ the mesh contains $1\,638\,400$ triangles and the
unreduced Galerkin problem has $817\,921$ interior nodal degrees of
freedom.

For fixed $R$, convergence as $\delta\to0$ follows from standard
non-self-adjoint Galerkin theory.  After adding a fixed positive shift, the
primal and adjoint source problems are coercive on $H^1_0(\Omega_R)$.  Since
$\Omega_R$ is a convex polygon and the coefficients are bounded, their
solution operators map $L^2(\Omega_R)$ boundedly into
$H^2(\Omega_R)\cap H^1_0(\Omega_R)$.  C\'ea's lemma and the piecewise affine
interpolation estimate give operator-norm convergence of the primal and
adjoint Galerkin solution operators.  The Osborn--Babu\v{s}ka spectral
approximation theorem therefore gives convergence to every isolated
eigenvalue of $H_R$, with algebraic multiplicity
\cite{osborn1975spectral,babuvska1991eigenvalue}.

The mass and stiffness matrices were integrated exactly.  The matrices
containing $V_{\rm per}+\Upsilon$ and $S$ were assembled using a
positive-weight triangle quadrature rule of order $10$.  At $R=20$, replacing
this rule by one of order $14$ changed the computed spectral set by
$2.7\times10^{-10}$ in Hausdorff distance on the coarsest mesh, and by
$1.61\times10^{-13}$ for $\delta=1/4$.  Thus quadrature error is negligible
compared with the discretisation error below.

For each $R$ and $\delta$ we computed the spectral set
\[
        \Sigma_{R,\delta}
        :=\{z\in\sigma(H_{R,\delta}):\operatorname{Re}z<1.3\}.
\]
For every discrete eigenvalue,
$\operatorname{Re}z\geq0$ and $0\leq\operatorname{Im}z\leq0.2$,
so the target set lies in a known compact rectangle $\{z:0\leq\operatorname{Re}z\leq1.3,0\leq\operatorname{Im}z\leq0.2\}$. We covered this
rectangle by two overlapping disk searches.  Each search retained at least
six exterior guard eigenvalues and was repeated with a larger Krylov space
and an independent starting vector.  The different searches agreed to within
$6\times10^{-14}$, and the largest relative eigenpair residual was
$6\times10^{-13}$.

Let $N_{R,\delta}$ denote the number of eigenvalues in $\Sigma_{R,\delta}$,
counted with algebraic multiplicity.  The counts
\[
\begin{array}{c|rrrrrr}
R  & 10 & 12 & 14 & 16 & 18 & 20\\
\hline
N_{R,\delta}  & 9 & 9 & 25 & 25 & 25 & 45
\end{array}
\]
were unchanged over all four values of $\delta$.  The finest-mesh spectral
sets are displayed in
\cref{fig:adaptive-spectra}.

\begin{figure}[t]
\centering
\includegraphics[width=\textwidth]{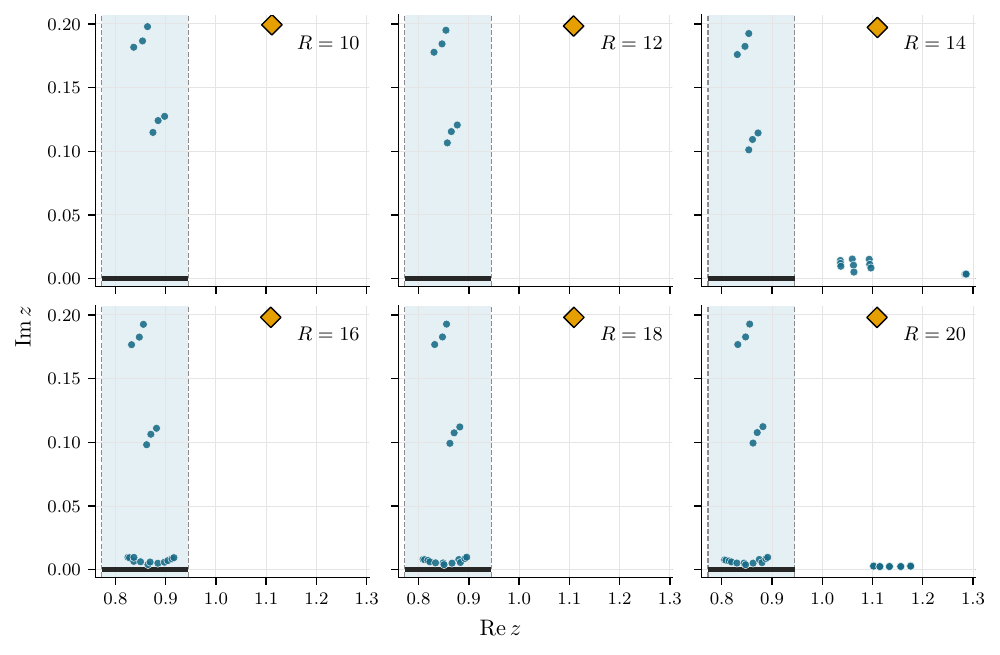}
\caption{The Galerkin eigenvalues with $\operatorname{Re}z<1.3$ for
$R=10,12,14,16,18,20$, computed on the finest mesh $\delta=1/16$.
All panels use the same axes.  The lightly shaded strip marks the real-part
window of the first periodic band $\mathcal I_1$, the black segment marks the
band itself, and the diamond marks the dissipative lift of the localised
defect eigenvalue.  Algebraic multiplicities were retained in the
computations, although coincident eigenvalues are displayed as one point.}
\label{fig:adaptive-spectra}
\end{figure}

To display convergence without treating the finest computed spectrum as an
exact reference, we use the successive refinement discrepancy
\[
 E_R(\delta)
   :=d_{\rm H}\bigl(\Sigma_{R,\delta},
                     \Sigma_{R,\delta/2}\bigr),
 \qquad
 \delta\in\left\{\frac12,\frac14,\frac18\right\}.
\]
Here $d_{\rm H}$ is the ordinary Hausdorff distance between finite subsets of
$\mathbb C$; it does not encode multiplicity.  As shown in
\cref{fig:mesh-convergence}, halving $\delta$ reduces every refinement
discrepancy by approximately a factor of four, consistent with quadratic convergence
of the computed eigenvalues.

\begin{figure}[t]
\centering
\includegraphics[width=.84\textwidth]{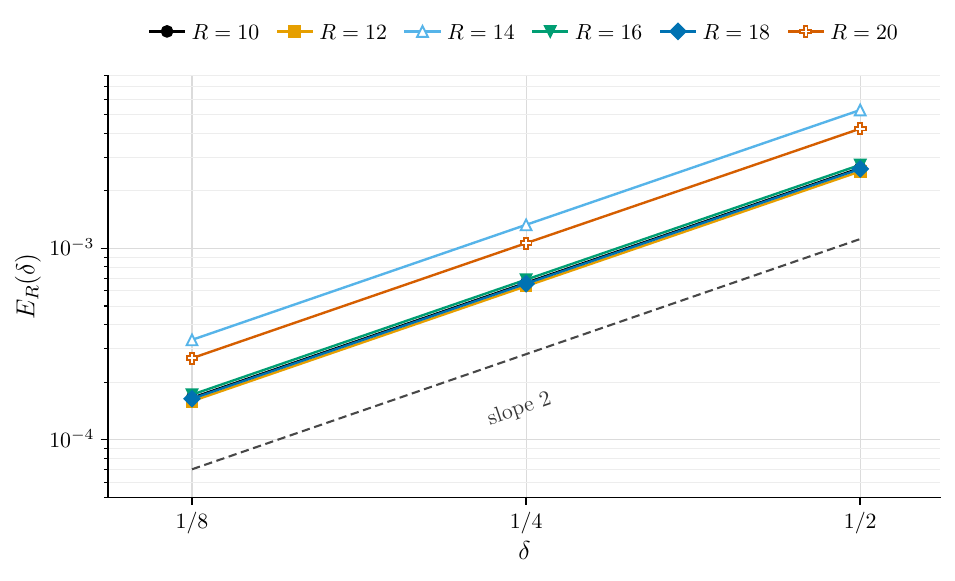}
\caption{Successive refinement discrepancies
$E_R(\delta)=d_{\rm H}(\Sigma_{R,\delta},\Sigma_{R,\delta/2})$ for
$R=10,12,14,16,18,20$.  The dashed reference line has slope $2$.}
\label{fig:mesh-convergence}
\end{figure}

\subsection{Lifted gap eigenvalue}

For each $R\in\{10,12,14,16,18,20\}$, let
$z_R^{\mathrm{def}}:=z_{R,1/16}^{\mathrm{def}}$ denote the lifted
defect eigenvalue marked by the diamond in
\cref{fig:adaptive-spectra}.  The values are shown in
\cref{tab:adaptive-gap-lift}.  For $R=16,18,20$ their variation is at
the $10^{-4}$ scale, comparable with the change between the two finest
meshes at $R=20$.  Thus the computations stabilise near
$
        z^{\mathrm{def}}\approx 1.1091+0.1981i.
$

\begin{table}[t]
\centering
\caption{The lifted defect eigenvalue
$z_R^{\mathrm{def}}=z_{R,1/16}^{\mathrm{def}}$.  The final column
measures the difference from the $R=20$ value and is a truncation
diagnostic rather than an error estimate.}
\begin{tabular}{cccc}
\toprule
$R$ & $\operatorname{Re}z_R^{\mathrm{def}}$
    & $\operatorname{Im}z_R^{\mathrm{def}}$
    & $\lvert z_R^{\mathrm{def}}-z_{20}^{\mathrm{def}}\rvert$\\
\midrule
10 & 1.111381367 & 0.199331613 & $2.63596\times10^{-3}$\\
12 & 1.108452795 & 0.198318990 & $6.53013\times10^{-4}$\\
14 & 1.109769649 & 0.197200077 & $1.13069\times10^{-3}$\\
16 & 1.109121343 & 0.198159111 & $9.86821\times10^{-5}$\\
18 & 1.109031919 & 0.198112127 & $4.24879\times10^{-5}$\\
20 & 1.109060928 & 0.198081084 & $0$\\
\bottomrule
\end{tabular}
\label{tab:adaptive-gap-lift}
\end{table}

To identify the self-adjoint state underlying this eigenvalue, let
$A_{R,\delta}$ denote the Galerkin operator obtained by deleting the term
$iS$ from $H_{R,\delta}$.  We performed targeted solves for
$A_{20,1/8}$ and $H_{20,1/8}$ on the same mesh.  They gave
\[
 \lambda_{20,1/8}^{\mathrm{def}}=1.109388494,
 \qquad
 z_{20,1/8}^{\mathrm{def}}
       =1.109158398+0.198078276i,
\]
and hence
\[
 \bigl|\operatorname{Re}z_{20,1/8}^{\mathrm{def}}
          -\lambda_{20,1/8}^{\mathrm{def}}\bigr|
       =2.30\times10^{-4}.
\]
The dissipative eigenvalue on this mesh differs from
$z_{20,1/16}^{\mathrm{def}}$ by $9.75\times10^{-5}$.  We use
$\delta=1/8$ for the mode comparison because it therefore resolves the
eigenvalue adequately while leaving the triangulation visible in the plot.

Let $u_0$ be the most localised self-adjoint gap eigenfunction found by this
solve, and let $u_S$ be the dissipative defect eigenfunction.  We normalise
both to have unit $L^2(\Omega_{20})$ norm, choose $u_0$ real, and multiply
$u_S$ by a unimodular constant so that
$
        \langle u_0,u_S\rangle_{L^2(\Omega_{20})}>0.
$
After this phase alignment,
\[
 \bigl|\langle u_0,u_S\rangle_{L^2(\Omega_{20})}\bigr|=0.9439.
\]
The respective fractions of $L^2$ mass in $\{\rho\leq8\}$ are
$0.8710$ and $0.9758$.  After restriction to this disk and
renormalisation, the overlap is $0.9997$.  Thus the localised cores are
nearly identical; the lower global overlap reflects a
truncation-induced exterior component of the self-adjoint mode.

Taking imaginary parts of the normalised Galerkin eigenvalue identity gives
\[
 \operatorname{Im}z_{20,1/8}^{\mathrm{def}}
   =\int_{\Omega_{20}}S|u_S|^2
   =0.198078276.
\]
Since $S=0.2$ on $\rho\leq8$, the concentration of the dissipative mode in
this region explains its nearly vertical displacement by $0.2i$ and the
close agreement of the localised profiles shown in \cref{fig:gap-modes}.

\begin{figure}[t]
\centering
\includegraphics[width=\textwidth]
  {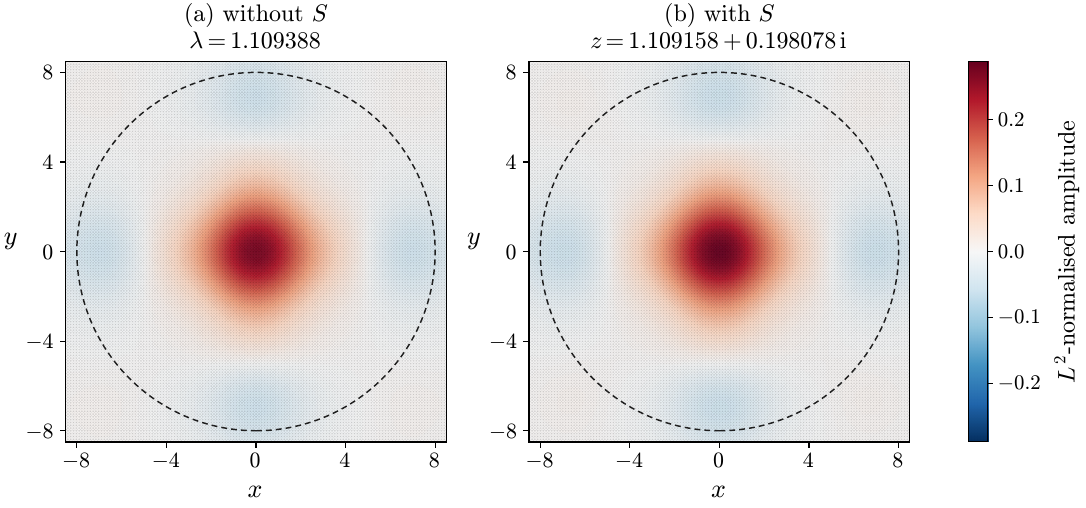}
\caption{The localised defect mode without and with the dissipative barrier,
computed on $\Omega_{20}$ with $\delta=1/8$.  Left: the real self-adjoint
eigenfunction $u_0$.  Right: $\operatorname{Re}u_S$, after phase alignment
with $u_0$.  Both eigenfunctions have unit $L^2(\Omega_{20})$ norm and the
panels use the same symmetric colour scale.  The grey lines show the
criss-cross finite element mesh and the dashed circle marks $\rho=8$, inside
which $S=0.2$.  Only the window $[-8.5,8.5]^2$ is displayed; the computation
used the full domain $\Omega_{20}$.}
\label{fig:gap-modes}
\end{figure}

\subsection{Absence of spectral invisibility}

For the square exhaustion $\Omega_R=(-R,R)^2$, the hypotheses of
\cref{thm:main} hold with $p=2$.  Consequently,
\[
 \mathcal I_1\cup\mathcal I_2
 =\sigma_{\rm ess}(H)
 \subseteq\sigma(H)
 \subseteq\liminf_{R\to\infty}\sigma(H_R).
\]
This concerns the continuous finite-domain operators.  Combining it with the
fixed-domain Galerkin convergence described above gives the corresponding
fully discrete conclusion: for every $E\in\mathcal I_1$ and
$\varepsilon>0$, there is $R_0$ such that, for each $R\geq R_0$,
\[
        \dist\bigl(E,\sigma(H_{R,\delta})\bigr)<\varepsilon
\]
once $\delta$ is sufficiently small.  The required mesh size may depend on
$R$.

To quantify the recovery of the first band in our computations, we use the
finest mesh and define the directed distance
\[
 D_R:=\sup_{E\in\mathcal I_1}
        \dist\bigl(E,\Sigma_{R,1/16}\bigr).
\]
Thus $D_R$ is the worst-case distance in $\mathbb C$ from the full interval
$\mathcal I_1$ to the complete computed spectral set in the window
$\operatorname{Re}z<1.3$.  The supremum is taken over the interval itself,
not over a sampling grid.

\begin{table}[t]
\centering
\caption{Directed distance from the first periodic band to the complete
finest-mesh spectral set, with $\delta=1/16$.}
\begin{tabular}{cc}
\toprule
$R$ & $D_R$\\
\midrule
10 & 0.153697\\
12 & 0.138050\\
14 & 0.129772\\
16 & 0.053431\\
18 & 0.050053\\
20 & 0.054368\\
\bottomrule
\end{tabular}
\label{tab:first-band-distance}
\end{table}
The worst-case distance therefore decreases by approximately a factor of
three over the displayed range of truncations.  The values need not be
monotone in $R$: in each row the maximum is attained at one endpoint of
$\mathcal I_1$, and the small increase from $R=18$ to $R=20$ reflects the
finite-volume placement of the nearest eigenvalue at that endpoint.  By
comparison, the successive-refinement distances between $\delta=1/8$ and
$\delta=1/16$ are at most $3.4\times10^{-4}$ in
\cref{fig:mesh-convergence}.

\section{Conclusion}

We have resolved the graveyard problem for dissipative Schr\"odinger
operators in dimensions $d\geq2$.  For
$
        H=-\Delta+V+iS,
$
with $V$ real and bounded below and with $S\geq0$ satisfying the natural
Sobolev-scale integrability assumptions, every nested bounded Dirichlet
exhaustion $\Omega_R\nearrow\Omega$ satisfies
$
        \sigma(H)
        \subseteq
        \liminf_{R\to\infty}\sigma(H_R).
$
Thus every neighborhood of every genuine spectral point of $H$ meets
$\sigma(H_R)$ for all sufficiently large $R$.  The result includes the
critical exponent $S\in L^{d/2}(\Omega)$ for $d\geq3$ and every
$S\in L^p(\Omega)$ with $p>1$ for $d=2$.  Together with the earlier
one-dimensional theorem \cite{marletta2014finite}, this settles the
no-invisibility question in all dimensions, within the corresponding
hypotheses.

The essential-spectrum case is the heart of the argument.  Generalized
strong resolvent convergence forces the self-adjoint truncations $A_R$ to
have increasingly many eigenvalues near each point of
$\sigma_{\rm ess}(A)$.  Cwikel-type estimates give a Schatten bound for
$H_R^{-1}-A_R^{-1}$ that is uniform in $R$, while the reversed
Hansmann--Weyl inequality prevents a uniformly controlled non-self-adjoint
perturbation from moving all of these eigenvalues away.  Spectral
invisibility is thereby converted into an eigenvalue-counting
contradiction.  This mechanism avoids the diagonal resolvent-kernel
estimates that obstruct the higher-dimensional extension of the original
one-dimensional proof, and replaces the variational arguments that are
unavailable for non-normal operators.

The theorem has a precise computational interpretation.  The
continuous domain-truncation stage cannot lose spectral points of the
dissipative target operator.  For each fixed computational domain,
standard Galerkin spectral approximation then controls the subsequent
finite-dimensional discretization.  The two-dimensional example
illustrates this separation of effects: a barrier lifts the localized gap state
into the upper half-plane, while the finite-domain spectra continue to
recover the first periodic band.  The observed mesh-refinement behavior
shows that both features are resolved by the finite element
discretization.

\medskip
\noindent\textbf{AI declaration.} Consistent with the Leiden Declaration on Artificial Intelligence \cite{LeidenDeclaration2026} and journal policy, we disclose all use of AI in this paper. As described in \cref{sec:route}, AI was used as an initial search aid. Its outputs were incorrect, but helped point us towards \cite{MR4764540,MR4722454}. No AI was used to produce the mathematical arguments or paper content, nor to draft the paper beyond minor language editing of human-written drafts. We take full responsibility for all content.

\medskip
\noindent\textbf{Acknowledgements.} The authors thank the Isaac Newton Institute for the Mathematical Sciences,
Cambridge, for support and hospitality during the {\em Geometric Spectral Theory and Applications} programme, where work on this paper was undertaken. This work was supported
by EPSRC grant EP/Z000580/1.

\bibliographystyle{abbrv}
\bibliography{pde_mcom}

\end{document}